\RequirePackage[l2tabu, orthodox]{nag}

\documentclass[12pt,reqno]{amsart}

\usepackage{fullpage,url,amssymb,enumerate,colonequals}
\usepackage[all]{xy} 
\usepackage{mathrsfs} 
\usepackage{comment}

\usepackage[OT2,T1]{fontenc}
\DeclareSymbolFont{cyrletters}{OT2}{wncyr}{m}{n}
\DeclareMathSymbol{\Sha}{\mathalpha}{cyrletters}{"58}

\usepackage{color}

\newcommand{\F}{\mathbb{F}}

\DeclareMathOperator{\Gal}{Gal}

\newcommand{\To}{\longrightarrow}

\newcommand{\orderofN}{b}

\newtheorem{theorem}{Theorem}[section]
\newtheorem{lemma}[theorem]{Lemma}
\newtheorem{corollary}[theorem]{Corollary}

\theoremstyle{definition}

\theoremstyle{remark}
\newtheorem{remark}[theorem]{Remark}

\makeatletter
\g@addto@macro\bfseries{\boldmath} 
\makeatother

\usepackage{microtype}

\usepackage[
	backref,
	pdfauthor={Bjorn Poonen, Kaloyan Slavov}, 
	pdftitle={The proportion of derangements characterizes the symmetric group},
]{hyperref}
\usepackage[alphabetic,backrefs,lite,nobysame]{amsrefs} 

\makeatletter
\@namedef{subjclassname@2020}{%
  \textup{2020} Mathematics Subject Classification}
\makeatother

\begin{document}

\title{The proportion of derangements characterizes the symmetric and alternating groups}



\author{Bjorn Poonen}
\thanks{B.P.\ was supported in part by National Science Foundation grants DMS-1601946 and DMS-2101040 and Simons Foundation grants \#402472 and \#550033.  K.S.\ was supported by NCCR SwissMAP of the SNSF}
\address{Department of Mathematics, Massachusetts Institute of Technology, Cambridge, MA 02139-4307, USA}
\email{poonen@math.mit.edu}
\urladdr{\url{http://math.mit.edu/~poonen/}}

\author{Kaloyan Slavov}
\address{Department of Mathematics, ETH Z\"{u}rich,
R\"{a}mistrasse 101, 8006 Z\"{u}rich, Switzerland
}
\email{kaloyan.slavov@math.ethz.ch}

\date{October 18, 2021}

\begin{abstract}
Let $G$ be a subgroup of the symmetric group $S_n$.  If the proportion of fixed-point-free elements in $G$ (or a coset) equals the proportion of fixed-point-free elements in $S_n$, then $G=S_n$. The analogue for $A_n$ holds if $n \ge 7$.  We give an application to monodromy groups. 
\end{abstract}

\maketitle

\section{Introduction}

\subsection{Derangements in permutation groups}

Motivated by an application to monodromy groups, we prove the following.

\begin{theorem}
\label{T:main}
Let $G$ be a subgroup of the symmetric group $S_n$ for some $n \ge 1$.
Let $C$ be a coset of $G$ in $S_n$.
If
\begin{equation}
\label{E:C}
\frac{|\{\sigma\in C : \sigma\ \textup{has no fixed points}\}|}{|C|}=\frac{|\{\sigma\in S_n : \sigma\ \textup{has no fixed points}\}|}{|S_n|},
\end{equation}
then $G=C=S_n$.
\end{theorem}

Elements of $S_n$ with no fixed points are called derangements.
Let $D_n$ be the number of derangements in $S_n$.
The right side of \eqref{E:C} is 
\[
    \frac{D_n}{n!} = \sum_{i=0}^n \frac{(-1)^i}{i!};
\]
see \cite{Stanley2012}*{Example~2.2.1}, for instance.
When the denominator of $D_n/n!$ in lowest terms is $n!$, 
the conclusion of Theorem~\ref{T:main} follows immediately, 
but controlling $\gcd(D_n,n!)$ in general is nontrivial. 
Our proof requires an irrationality measure for $e$, 
divisibility properties of $D_n$, 
and a bound on the orders of primitive permutation groups.

\begin{remark}
The proof shows also that for $n \ge 5$, if $C$ is not necessarily a coset but just any subset of $S_n$ having the same size as $G$, then \eqref{E:C} implies that $G$ is $A_n$ or $S_n$.  In fact, we prove that if a subgroup $G$ of $S_n$ has order divisible by the denominator of $D_n/n!$, then $G$ is $A_n$ or $S_n$.
\end{remark}

\begin{remark}
We also prove an analogue of Theorem~\ref{T:main} in which both appearances of $S_n$ on the right side of \eqref{E:C} are replaced by the alternating group $A_n$ for some $n \ge 7$; see Theorem~\ref{T:alternating}. 
But there are counterexamples for smaller alternating groups.
For example, the order~$10$ dihedral group in $A_5$ has the same proportion of derangements as $A_5$, namely $4/10=24/60$.
\end{remark}

\subsection{Application to monodromy}

Let $\F_q$ be the finite field of $q$ elements.
Let $f(T)\in\F_q[T]$ be a polynomial of degree $n$.
Birch and Swinnerton--Dyer \cite{Birch-Swinnerton-Dyer1959} define what it means for $f$ to be ``general'' and estimate the proportion of field elements in the image of a general $f$:
\[
   \frac{|f(\F_q)|}{q} = 1 - \sum_{i=0}^n\frac{(-1)^i}{i!} + O_n(q^{-1/2}).
\]

More generally, let $f\colon X\to Y$ be a degree~$n$ generically \'etale morphism of schemes of finite type over $\F_q$, with $Y$ geometrically integral.
The geometric and arithmetic monodromy groups $G$ and $A$ are subgroups of $S_n$ fitting in an exact sequence
\[
    1 \To G \To A \To \Gal(\F_{q^r}/\F_q) \To 1
\]
for some $r \ge 1$; see \cite{Entin2021}*{Section~4} for an exposition.
Let $C$ be the coset of $G$ in $A$ mapping to the Frobenius generator of $\Gal(\F_{q^r}/\F_q)$.
Let $M$ be a bound on the geometric complexity of $X$ and $Y$.
Assume that $Y(\F_q) \ne \emptyset$, which is automatic if $q$ is large relative to $M$.
Then the Lang--Weil bound implies  
\begin{equation}
\label{E:meaning of fraction}
\frac{|f(X(\F_q))|}{|Y(\F_q)|}=\frac{|\{\sigma\in C : \sigma\ \textup{has at least one fixed point}\}|}{|C|}+O_{n,M}(q^{-1/2});
\end{equation}
see \cite{Entin2021}*{Theorem 3}, for example.
In particular, if $G=S_n$, then
\begin{equation}
\frac{|f(X(\F_q))|}{|Y(\F_q)|}=1-\sum_{i=0}^n\frac{(-1)^i}{i!}+O_{n,M}(q^{-1/2}).
\label{E:image_point_count}
\end{equation}

We prove a \emph{converse}, that an estimate as
in \eqref{E:image_point_count}
on the proportion of points in the image
implies that the geometric monodromy group of $f$ is the full symmetric group $S_n$:

\begin{corollary}
Given $n$ and $M$, there exists an effectively computable constant $c=c(n,M)$ such that for any $f \colon X \to Y$
as above, with $\deg f=n$ and the complexities of $X$ and $Y$ bounded by $M$, if 
\[
    \frac{|f(X(\F_q))|}{|Y(\F_q)|} = 1 - \sum_{i=0}^n \frac{(-1)^{i}}{i!} + \epsilon, \quad \text{where $|\epsilon| < \frac{1}{n!} - c q^{-1/2}$,}
\]
then $G=S_n$. 
\label{Cor:application}
\end{corollary}

\begin{proof}
Combine \eqref{E:meaning of fraction} and Theorem~\ref{T:main}.
\end{proof}

\begin{remark}
We originally proved Corollary~\ref{Cor:application} in order to prove a version of \cite{Poonen-Slavov2020}*{Theorem~1.9}, about specialization of monodromy groups, but later we found a more natural argument.
\end{remark}

\subsection{Structure of the paper}
The proof of Theorem~\ref{T:main} occupies the rest of the paper, which is divided in sections
according to the properties of $G$.
Throughout, we assume that $G$, $C$, and $n$ are such that \eqref{E:C} holds.
The cases with $n \le 4$ can be checked directly,
so assume that $n \ge 5$ and $G \ne S_n$.

\section{Primitive permutation groups}
\label{S:primitive}

The proportion of derangements in $A_n$ is given by the inclusion-exclusion formula; 
it differs from $D_n/n!$ by the nonzero quantity $\pm(n-1)/n!$. 
The proportion for $S_n$ is the average of the proportions for $A_n$ and $S_n-A_n$,
so the proportion for $S_n-A_n$ also differs from $D_n/n!$.
Thus $G \ne A_n$.

Suppose that $G$ is primitive, $n \ge 5$, and $G \ne A_n,S_n$.
The main theorem in \cite{Praeger-Saxl1980}\footnote{This is independent of the classification of finite simple groups.  Using the classification, \cite{Maroti2002} gives better bounds.} 
gives $|G|<4^n$.
On the other hand, $D_n/n!$ is close to $1/e$ and hence cannot equal a rational number with small denominator; this will show that $|G|$ is at least about $\sqrt{n!}$.
These will give a contradiction for large $n$.
We now make this precise.

Let $a=|\{\sigma\in C : \sigma\ \textup{has no fixed points}\}|$ and $\orderofN=|C|=|G|$, so $a \le \orderofN = |G| < 4^n$.
Then
\[\left|\frac{a}{\orderofN}-\frac{1}{e}\right|=\left|\frac{D_n}{n!}-\frac{1}{e}\right|<\frac{1}{(n+1)!}.\]
No rational number with numerator $\le 4$ is within $1/6!$ of $1/e$, so $a \ge 5$.
By the main result of \cite{Okano1992} (see also \cite{Alzer1998}),
\[
   \left| e - \frac{\orderofN}{a} \right| > \frac{\log \log a}{3 a^2 \log a}.
\]
Combining the two displayed inequalities yields
\begin{equation}
\label{E:big inequality}
   \frac{1}{(n+1)!} > \left|\frac{a}{\orderofN}-\frac{1}{e}\right|
    = \frac{a}{\orderofN e}\left|e-\frac{\orderofN}{a}\right|
    > \frac{1}{\orderofN e}\cdot\frac{\log\log a}{3a\log a}
    > \frac{\log \log 4^n}{3 e (4^n)^2 \log 4^n};
\end{equation}
the last step uses that $a,\orderofN < 4^n$ and 
that $\dfrac{\log\log x}{x\log x}$ 
is decreasing for $x\geq 5$. 
Inequality~\eqref{E:big inequality} implies $n \le 41$.

Let $d_n$ be the denominator of the rational number $\dfrac{D_n}{n!} = \dfrac{a}{\orderofN}$.
Then $d_n \mid \orderofN$, so $d_n \le \orderofN < 4^n$.
For $11 < n \le 41$, the inequality $d_n < 4^n$ fails.
For $n \le 11$, a Magma computation \cite{Magma} shows that there are no degree $n$ primitive subgroups $G \ne A_n,S_n$ for which $d_n \mid \orderofN$.

\section{Imprimitive but transitive permutation groups}
\label{S:imprimitive}

Suppose that $G$ is imprimitive but transitive.
Then $G$ preserves a partition of $\{1,\ldots,n\}$ into $l$ subsets of equal size $k$, 
for some $k,l \ge 2$ with $kl=n$.  
The subgroup of $S_n$ preserving such a partition has order $(k!)^l l!$ (it is a wreath product $S_k \wr S_l$).
Thus $|G|$ divides $(k!)^l l!$.

For a prime $p$, let $\nu_p$ denote the $p$-adic valuation.
Since $\dfrac{a}{|G|} = \dfrac{D_n}{n!}$, every prime $p \nmid D_n$ satisfies $\nu_p(n!) \le \nu_p(|G|) \le \nu_p((k!)^l l!) \le \nu_p(n!)$.
Thus for every prime $p\nmid D_n$, the inequality
$ \nu_p((k!)^l l!) \le \nu_p(n!)$
is an equality.
The third of the three following lemmas will prove that this is impossible for $n \ge 5$.

\begin{lemma}
Let $k, l\geq 2$ and let $p$ be a prime. The inequality
\begin{equation}
\nu_p((k!)^l l!)\leq \nu_p((kl)!)
\label{nu_p_factorial_inequality}
\end{equation}
is an equality if and only if at least one of the following holds:
\begin{itemize}
\item $k$ is a power of $p$;
\item there are no carry operations in the $l$-term addition $k+\cdots+k$ when $k$ is written in base $p$  $($in particular, $l<p$$)$. 
\end{itemize}
\label{equality_case}
\end{lemma} 

\begin{proof}
Let $s_p(k)$ denote the sum of the $p$-adic digits of a positive integer $k$; then $\nu_p(k!)=\dfrac{k-s_p(k)}{p-1}$.
Thus equality in \eqref{nu_p_factorial_inequality} is equivalent to equality in
\begin{equation}
l+s_p(kl)\leq ls_p(k)+s_p(l).
\label{nu_p_digits_inequality}
\end{equation} 
We always have
\begin{equation}
    l+s_p(kl)\leq l+s_p(k)s_p(l)\leq ls_p(k)+s_p(l);
\label{E:two steps}
\end{equation}
the first follows from $s_p(kl)\leq s_p(k)s_p(l)$, 
and the second is simply 
\[
   (s_p(k)-1)(l-s_p(l))\geq 0.
\]
Thus equality in~\eqref{nu_p_digits_inequality} is equivalent to equality in both inequalities of~\eqref{E:two steps}.

The second inequality of~\eqref{E:two steps} is an equality 
if and only if either $k$ is a power of $p$ or $l<p$;
in each case, we must check when equality holds in the first inequality~\eqref{E:two steps},
i.e., when $s_p(kl) = s_p(k) s_p(l)$.
If $k$ is a power of $p$, then it holds. 
If $l<p$, then it holds if and only if $s_p(kl) = l s_p(k)$,
which holds if and only if there are no carry operations in the $l$-term addition 
$k+\cdots+k$ when $k$ is written in base $p$.
\end{proof}

The following lemma will help us produce primes $p$ not dividing $D_n$.

\begin{lemma} 
\label{L:supply_primes}
For $0 \le m \le n$, we have $D_n \equiv (-1)^{n-m} D_m \pmod{n-m}$.
In particular, 
\begin{align}
\label{E:D mod n} D_n &\equiv \pm 1\pmod{n}\\
\label{E:D mod n-2} D_n &\equiv \pm 1\pmod{n-2}\\
\label{E:D mod n-3} D_n &\equiv \pm 2\pmod{n-3}.
\end{align}
\end{lemma}

\begin{proof}
Reduce each term in $D_n$ modulo $n-m$; most of them are $0$.
\end{proof}

\begin{lemma}
Let $k,l\geq 2$. Set $n=kl$ and assume $n>4$. Then there exists a prime $p\nmid D_n$ such that
\[\nu_p((k!)^l l!)< \nu_p(n!).\]
\label{L:imprimitive_main_lemma}
\end{lemma}

\begin{proof}

{\bfseries Case 1.  $l\geq 3$ and $n-2$ is not a power of $2$.}

Let $p\geq 3$ be a prime with $p\mid n-2$. 
By \eqref{E:D mod n-2}, $p\nmid D_n$, so $\nu_p((k!)^l l!) = \nu_p(n!)$.
Apply Lemma~\ref{equality_case}.
If $k$ is a power of $p$, then $p$ divides $k$, which divides $n$, 
so $p\mid n-(n-2)=2$, contradicting $p \ge 3$. 
Otherwise, there are no carry operations in the $l$-term addition $k+\cdots+k$ in base $p$.
This is impossible because the last digit of $n$ is $2$ (since $p\mid n-2$ and $p\geq 3$) and $l\geq 3$. 

\medskip

{\bfseries Case 2.  $l=2$.}

Then $2 \mid n$.
By \eqref{E:D mod n}, $2 \nmid D_n$.
By Lemma \ref{equality_case}, $k$ is a power of $2$ (since $l<2$ is violated).
Thus $n=2k$ is a power of $2$. 

Since $n \ge 5$, there exists a prime $p \mid n-3$. 
Since $n$ is a power of $2$, this implies $p \ge 5$.
By \eqref{E:D mod n-3}, $p\nmid D_n$. Apply Lemma~\ref{equality_case}. Note that $k$ is not a power of $p$, since $k$ is a power of $2$ and $p\neq 2$. Therefore, there are no carry operations in $k+k=n$,
so the last digit of $n$ is even.
But $p \mid n-3$ and $p \ge 5$, so the last digit of $n$ is $3$.

\medskip

{\bfseries Case 3. $l=3$ and $n-2$ is a power of $2$.}

Then $3\mid n$.
By \eqref{E:D mod n}, $3 \nmid D_n$.
By Lemma~\ref{equality_case}, $k$ must be a power of $3$ (since $l<3$ is violated). 
Then $n=3k$ is a power of $3$, contradicting the fact that $n$ is even. 

\medskip

{\bfseries Case 4. $l>3$ and $n-2$ is a power of $2$.}

In particular, $n = kl > 6$. 
Then $n-3$ is not a power of $3$, 
because otherwise we would have a solution to $3^u=2^v-1$ with $u>1$,
whereas the only solution in positive integers is $(u,v)=(1,2)$
(proof: $3\mid 2^v-1$, so $v$ is even, so $2^{v/2}-1$ and $2^{v/2}+1$ are powers of $3$ that differ by $2$, so they are $1$ and $3$).  

Let $p\neq 3$ be a prime divisor of $n-3$. 
Then $p\geq 5$. 
Apply \eqref{E:D mod n-3} and Lemma~\ref{equality_case}. 
If $k$ is a power of $p$, then $p\mid n$, so $p \mid n-(n-3)=3$, contradicting $p \ne 3$. 
Therefore, there are no carry operations in the $l$-term addition $k+\dots+k$. This is impossible, since the last digit of $kl$ is $3$ (since $p\mid n-3$ and $p\geq 5$) and $l>3$.
\end{proof}

\section{Intransitive permutation groups}
\label{S:intransitive}

Suppose that $G$ is intransitive.
Then $G$ embeds in $S_u \times S_v \subset S_n$ for some $u,v\ge 1$ with $u+v=n$.

Consider a prime $p\mid n$. By \eqref{E:D mod n}, $p\nmid D_n$.
Then, analogously to the second paragraph of Section~\ref{S:imprimitive}, 
$\nu_p(n!) \le \nu_p(|G|) \le \nu_p(u! \, v!) \le \nu_p(n!)$,
so $\nu_p(u!)  + \nu_p(v!) = \nu_p(n!)$; equivalently, $s_p(u) + s_p(v) = s_p(n)$.
So there are no carry operations in $u+v$. 
Let $e=\nu_p(n)$, so the last $e$ base $p$ digits of $n$ are zero;
then the same holds for $u$ and $v$.
In other words, $p^e\mid u,v$ as well. 
Since this holds for each $p\mid n$, we conclude that $n\mid u,v$.
This contradicts $0 < u,v < n$.

This completes the proof of Theorem~\ref{T:main}.


\section{Alternating group}

\begin{theorem}
Let $G$ be a subgroup of the symmetric group $S_n$ for some $n\geq 7$. Let $C$ be a coset of $G$ in $S_n$ having the same proportion of fixed-point-free elements as $A_n$. Then $G=A_n$.
\label{T:alternating}
\end{theorem}

\begin{remark} 
For $n\leq 6$, the subgroups of $S_n$ other than $A_n$ for which some coset has the same proportion as $A_n$, up to conjugacy, are
\begin{itemize}
\item the order~$4$ subgroup of $S_4$ generated by $(1423)$ and $(12)(34)$;
\item the order~$4$ subgroup of $S_4$ generated by $(34)$ and $(12)(34)$;
\item the order~$8$ subgroup of $S_4$;
\item the subgroups of $S_5$ of order $5$, $10$, or $20$;
  \item the order~$36$ subgroup of $S_6$ generated by $(1623)(45)$, $(12)(36)$, $(124)(365)$, and $(142)(365)$;
 \item the order~$36$ subgroup of $S_6$ generated by $(13)(25)(46)$, $(14)(36)$, $(154)(236)$, and $(145)(236)$.
  \end{itemize}
\end{remark}

The proof of Theorem~\ref{T:alternating} follows the proof of Theorem~\ref{T:main};
we highlight only the differences.
The proportion of fixed-point-free elements in $A_n$ is $E_n/n!$, where 
$E_n\colonequals D_n+(-1)^{n-1}(n-1)$.

\subsection{Primitive permutation groups}
Suppose $G\neq A_n$.
The first paragraph of Section~\ref{S:primitive} shows that $G \ne S_n$.
For $7 \le n \le 13$, we use Magma to check Theorem~\ref{T:alternating} for each primitive subgroup of $S_n$.
So assume $n \ge 14$.  
Define $a$ and $b$ as in Section~\ref{S:primitive}.
We have
\[
\left|\frac{a}{b}-\frac{1}{e}\right| = \left| \frac{E_n}{n!} - \frac{1}{e} \right| \le \left| \frac{E_n-D_n}{n!} \right| + \left| \frac{D_n}{n!} - \frac{1}{e} \right|< 
\frac{n-1}{n!} + \frac{1}{(n+1)!} =\frac{n^2}{(n+1)!}.
\]
No $a/b$ with $a<5$ is within $15^2/16!$ of $1/e$, so $a \ge 5$.
Inequality~\eqref{E:big inequality} with $1/(n+1)!$ replaced by $n^2/(n+1)!$ implies $n \le 49$.

Let $e_n$ be the denominator of $E_n/n!$, so $e_n$ divides $|G|$, which is less than $4^n$.
But for $13<n\leq 49$, the inequality $e_n<4^n$ fails.

\subsection{Imprimitive permutation groups that preserve a partition into blocks of equal size}
\label{Sub:imprimitive}

To rule out imprimitive permutation groups that preserve a partition into $l$ blocks of size $k$, 
we argue as in Section~\ref{S:imprimitive}, but with Lemma~\ref{L:imprimitive_main_lemma} replaced by the following.

\begin{lemma}
Let $k,l\geq 2$. Set $n=kl$ and assume that $n>6$. Then there exists a prime $p\nmid E_n$ such that
\[\nu_p((k!)^l l!)<\nu_p(n!).\]
\label{Lem:strict_ineq}
\end{lemma}

\begin{proof}[Proof of Lemma \ref{Lem:strict_ineq}]
For each integer $n\in (6,30]$, we check directly that there exists a prime $p\in (n/2,n]$ such that $p\nmid E_n$. Assume from now on that $n>30$.

Suppose the statement is false. Then whenever a prime $p$ satisfies $p\nmid E_n$, \eqref{nu_p_factorial_inequality} is an equality and  Lemma \ref{equality_case} applies. 

By using $D_n\equiv (-1)^{n-s}D_s\pmod{n-s}$ and $E_n=D_n+(-1)^{n-1}(n-1)$, we obtain
\begin{align}
 \label{E:n}   E_n &\equiv 2(-1)^n\pmod{n}\\
 \label{E:n-3}      E_n &\equiv 4(-1)^{n-1}\pmod{n-3}\\
 \label{E:n-4}      E_n &\equiv 6(-1)^n\pmod{n-4}\\
 \label{E:n-5}      E_n &\equiv (-1)^{n-1}2^4\times 3\pmod{n-5}
\end{align}

{\bfseries Case 1. $n-4$ is a power of $2$.}

Then $n-3$ is not a power of $3$ because otherwise, we have a solution to $3^u-1=2^v$ with $u \ge 3$;
working modulo~$4$ shows that $u$ is even, and factoring the left side leads to a contradiction.
Let $p\neq 3$ be a prime with $p\mid n-3$. Since $n-3$ is odd, $p\geq 5$. By \eqref{E:n-3}, $p\nmid E_n$, so we have one of the conclusions of Lemma~\ref{equality_case}.

If $k$ is a power of $p$, then $p\mid k\mid n$, which, combined with $p\mid n-3$ gives $p=3$, a contradiction. 

Suppose that there is no carry in $k+\cdots+k$ ($l$ terms). This sum has last digit $3$ in base $p$, so $l=3$, so $3\mid n$, and hence $3\nmid E_n$ by \eqref{E:n}. Apply Lemma \ref{equality_case} for the prime $3$. Since $l<3$ is violated, we deduce that $k$ is a power of $3$. Then $n=kl$ is also a power of $3$, but this contradicts the fact that $n$ is even.

\bigskip

{\bfseries Case 2. $n-3$ is a power of $2$ and $l\neq 2,4$.}

Then $n-4$ is odd and is not a power of $3$. Let $p\neq 3$ be a prime with $p\mid n-4$. Then $p\geq 5$, so $p\nmid E_n$ by \eqref{E:n-4}. If $k$ is a power of $p$, then $p\mid k\mid n$, which contradicts $p\mid n-4$ since $p \ge 5$. If there are no carry operations in the $l$-term addition $k+\cdots+k$ (which has last digit $4$ in base $p$), then $l=2$ or $l=4$, contrary to assumption.

\bigskip

{\bfseries Case 3. $l=3$.} 

Then $3\mid n$, hence $3\nmid E_n$ by \eqref{E:n}. Apply Lemma \ref{equality_case} for the prime $3$. Since $3<l$ is violated, $k$ is a power of $3$. 
Then $n=kl$ is also a power of $3$. Then $n-4$ is odd and not divisible by $3$. Let $q$ be a prime with 
$q\mid n-4$. Then $q\geq 5$, and hence $q\nmid E_n$ by
\eqref{E:n-4}. Since $k$ is a power of $3$, it is not a power of $q$. So there is no carry in $k+k+k$ in base $q$. But this sum has last digit $4$ in base $q$, which is a contradiction. 

\bigskip

{\bfseries Case 4. $l\neq 2,4$.} 

By the previous cases, we may assume in addition that $n-4$ and $n-3$ are not powers of $2$ and $l\neq 3$.

Let $p\neq 2$ be a prime with $p\mid n-3$. Then $p\nmid E_n$ by \eqref{E:n-3}.
Since the $l$-term addition $k+\cdots+k$ has last digit $3$ and $l\neq 3$, there is some carry. Therefore $k$ is a power of $p$. Then $p\mid k\mid n$, which, combined with $p\mid n-3$, gives 
$p=3$. In particular, $3\mid n$.

Let $q\neq 2$ be a prime with $q\mid n-4$. Since $3 \mid n$, we have $q\neq 3$ so $q\geq 5$. By \eqref{E:n-4}, $q\nmid E_n$. If $k$ is a power of $q$, then $q\mid n$, hence $q \mid 4$ --- contradiction. Therefore there is no carry in the $l$-term addition $k+\cdots+k$ in base $q$. This sum has last digit $4$ and $l\neq 2,4$, so this case is impossible.

\bigskip

{\bf Case 5. $l=2$ or $l=4$.} 

Then $n$ is even, so $n-3$ and $n-5$ are odd. 

\medskip

{\em Subcase 5.1: $n-3$ is not a power of $3$.}

Let $p\neq 3$ be a prime such that $p\mid n-3$. Then $p\geq 5$ and $p\nmid E_n$ by \eqref{E:n-3}. If $k$ is a power of $p$, then $p\mid k\mid n$, giving $p=3$, which is a contradiction. However, there is carry in the $l$-term addition $k+\cdots+ k$ because the sum has last digit $3$, and $l$ is $2$ or $4$.

\medskip

{\em Subcase 5.2: $n-3$ is a power of $3$ but $n-5$ is not a power of $5$.}

Let $p\neq 5$ be a prime with $p\mid n-5$. Then $p\geq 7$ and we apply the argument of subcase~5.1: an $l$-term sum $k+\cdots+k$ cannot have last digit $5$ in base $p$. 

\medskip

{\em Subcase 5.3: $n-3=3^a$ and $n-5=5^b$ for some $a,b\geq 1$.}

Then $3^a-5^b=2$, so $a=3$ and $b=2$ by \cite{Brenner-Foster}*{Theorem~4.06}.
This contradicts $n > 30$.
\end{proof}

\subsection{Intransitive subgroups}

As in Section \ref{S:intransitive}, $G$ embeds in 
$S_u \times S_v \subset S_n$ for some $u,v\ge 1$ with $u+v=n$.
Write $n=2^s m$, where $s\geq 0$ and $2\nmid m$. The argument in Section~\ref{S:intransitive} for odd $p$ with $E_n$ in place of $D_n$ and \eqref{E:n} in place of \eqref{E:D mod n} implies that $m\mid u,v$. Thus $s\geq 1$.

If $s=1$, then $n=2m$, so $u=v$. This case is covered in Section~\ref{Sub:imprimitive}.

Suppose that $s\geq 2$. Then $4 \mid n$, so \eqref{E:n} implies that $E_n/2$ is odd. Using
$\frac{a}{|G|}=\frac{E_n/2}{n!/2}$, we obtain
$\nu_2(n!/2)\leq \nu_2(|G|)\leq \nu_2(u!v!)\leq \nu_2(n!)$.
If the last inequality is an equality,
then the same argument used in Section~\ref{S:intransitive} shows that $\nu_2(u) = \nu_2(v) = \nu_2(n)$;
combining this with $m \mid u,v$ shows that $n \mid u,v$, a contradiction.
Therefore the first two inequalities must be equalities, so $\nu_2(u!v!)=\nu_2(n!)-1$; equivalently, $s_2(u)+s_2(v)=s_2(n)+1$. This means there is exactly one carry operation in $u+v$ in base $2$. This is possible only when $2^{s-1}\mid u,v$. Also, $m \mid u,v$, so $n/2 \mid u,v$, so again $u=v$, and this case is covered in Section~\ref{Sub:imprimitive}.

\section*{Acknowledgements}
We thank Andrew Sutherland for useful discussions concerning Section~\ref{S:primitive} and specifically for drawing our attention to~\cite{Okano1992}.
We thank Michael Bennett and Samir Siksek for suggesting references for the solution of $3^a-5^b=2$.
We also thank the referees for comments.


\begin{bibdiv}
\begin{biblist}


\bib{Alzer1998}{article}{
   author={Alzer, Horst},
   title={On rational approximation to $e$},
   journal={J. Number Theory},
   volume={68},
   date={1998},
   number={1},
   pages={57--62},
   issn={0022-314X},
   review={\MR{1492888}},
   doi={10.1006/jnth.1997.2199},
}

\bib{Birch-Swinnerton-Dyer1959}{article}{
   author={Birch, B. J.},
   author={Swinnerton-Dyer, H. P. F.},
   title={Note on a problem of Chowla},
   journal={Acta Arith.},
   volume={5},
   date={1959},
   pages={417--423},
   issn={0065-1036},
   review={\MR{113844}},
   doi={10.4064/aa-5-4-417-423},
}


\bib{Brenner-Foster}{article}{
   author={Brenner, J. L.},
   author={Foster,  Lorraine L.},
   title={Exponential Diophantine equations},
   journal={Pacific J. Math.},
   volume={101},
   date={1982},
   number={2},
   pages={263--301},
   issn={0030-8730}, 
   review={\MR{675401}},
   doi={10.2140/pjm.1982.101.263},
}

\bib{Entin2021}{article}{
   author={Entin, Alexei},
   title={Monodromy of hyperplane sections of curves and decomposition
   statistics over finite fields},
   journal={Int. Math. Res. Not. IMRN},
   date={2021},
   number={14},
   pages={10409--10441},
   issn={1073-7928},
   review={\MR{4285725}},
   doi={10.1093/imrn/rnz120},
}

\bib{Magma}{article}{
    author={Bosma, Wieb},
    author={Cannon, John},
    author={Playoust, Catherine},
     title={The Magma algebra system. I. The user language},
      note={Computational algebra and number theory (London, 1993).  Magma is available at \url{http://magma.maths.usyd.edu.au/magma/}\phantom{i}}, 
   journal={J. Symbolic Comput.},
    volume={24},
      date={1997},
    number={3-4},
     pages={235\ndash 265},
      issn={0747-7171},
    review={\MR{1484478}},
     label={Magma}, 
}

\bib{Maroti2002}{article}{
   author={Mar\'{o}ti, Attila},
   title={On the orders of primitive groups},
   journal={J. Algebra},
   volume={258},
   date={2002},
   number={2},
   pages={631--640},
   issn={0021-8693},
   review={\MR{1943938}},
   doi={10.1016/S0021-8693(02)00646-4},
}


\bib{Okano1992}{article}{
   author={Okano, Takeshi},
   title={A note on the rational approximations to $e$},
   journal={Tokyo J. Math.},
   volume={15},
   date={1992},
   number={1},
   pages={129--133},
   issn={0387-3870},
   review={\MR{1164191}},
   doi={10.3836/tjm/1270130256},
}

\bib{Poonen-Slavov2020}{article}{
   author={Poonen, Bjorn},
   author={Slavov, Kaloyan},
   title={The exceptional locus in the Bertini irreducibility theorem for a morphism},
   journal={Int.\ Math.\ Res.\ Notices},
   volume={rnaa182},
   date={2020-08-04},
   issn={1687-0247},
   doi={10.1093/imrn/rnaa182},
}

\bib{Praeger-Saxl1980}{article}{
   author={Praeger, Cheryl E.},
   author={Saxl, Jan},
   title={On the orders of primitive permutation groups},
   journal={Bull. London Math. Soc.},
   volume={12},
   date={1980},
   number={4},
   pages={303--307},
   issn={0024-6093},
   review={\MR{576980}},
   doi={10.1112/blms/12.4.303},
}

\bib{Stanley2012}{book}{
   author={Stanley, Richard P.},
   title={Enumerative combinatorics. Volume 1},
   series={Cambridge Studies in Advanced Mathematics},
   volume={49},
   edition={2},
   publisher={Cambridge University Press, Cambridge},
   date={2012},
   pages={xiv+626},
   isbn={978-1-107-60262-5},
   review={\MR{2868112}},
}

\end{biblist}
\end{bibdiv}
\end{document}